\documentclass[11pt, reqno]{amsart}
\usepackage[T1]{fontenc}
\usepackage{amsfonts}
\usepackage{amsmath}
\usepackage{amsthm}
\usepackage{amssymb}
\usepackage{geometry}
\usepackage{graphicx}
\usepackage{xcolor}
\usepackage{mathtools}
\usepackage[colorlinks=true, linkcolor=blue]{hyperref}
\usepackage{cleveref}
\usepackage[english]{babel}
\usepackage{lineno}
\usepackage{float}

\newtheorem{theorem}{Theorem}[section]

\newtheorem{lemma}[theorem]{Lemma}

\newtheorem{proposition}[theorem]{Proposition}
\newtheorem{problem}[theorem]{Problem}
\usepackage{tikz}
\usetikzlibrary{positioning}
\usetikzlibrary{decorations,arrows}
\usetikzlibrary{decorations.markings}
\numberwithin{equation}{section}

\newcommand{\R}{\mathbb R}

\usepackage{rustic}
\usepackage[T1]{fontenc}

\emergencystretch=\maxdimen
\title{Graph identification index}
\author{Runze Wang}
\address[]{Department of Mathematical Sciences, University of Memphis, Memphis, TN 38152, USA}
\email{rwang6@memphis.edu; runze.w@hotmail.com}
\thanks{}
\date{\today}
\subjclass[]{}

\begin{document}

\sloppy

\begin{abstract}
    We introduce the \emph{ID-index} of a finite simple connected graph. For a graph $G=(V,\ E)$ with diameter $d$, we let $f:V\longrightarrow \mathbb{R}$ assign \emph{ranks} to the vertices, then under $f$, each vertex $v$ gets a \emph{string}, which is a $d$-vector with the $i$-th coordinate being the sum of the ranks of the vertices that are of distance $i$ from $v$. The \emph{ID-index} of $G$, denoted by $IDI(G)$, is defined to be the minimum number $k$ for which there is an $f$ with $|f(V)|=k$, such that each vertex gets a distinct string under $f$. We present some relations between ID-graphs, which were defined by Chartrand, Kono, and Zhang, and their ID-indices; give a lower bound on the ID-index of a graph; and determine the ID-indices of paths, grids, cycles, prisms, complete graphs, some complete multipartite graphs, and some caterpillars.
\end{abstract}

\maketitle

\section{Introduction}
In this paper, every graph is assumed to be a finite simple \textbf{connected} graph, so each graph we study has a finite diameter.

The topic of uniquely identifying each vertex in a graph has been extensively studied. For example, we have the \emph{metric dimension} of a graph $G=(V,\ E)$, which is the smallest size of a set $S\subseteq V$ such that every vertex $v\in V$ can be uniquely determined by the distances between $v$ and the vertices in $S$. This concept was introduced by Slater \cite{Sl} in 1975 and independently by Harary and Melter \cite{HM} in 1976, and has been studied by multiple scholars (see \cite{BDF,CEJO,HMPSW,KRR}). Another example is the \emph{partition dimension} of a graph, where vertex colorings are used (see \cite{CGH,CSZ,MS,RGL}).

Chartrand, Kono, and Zhang introduced ID-graphs in \cite{CKZ}, which gave a new idea of identifying vertices in a graph. For a graph $G$ with diameter $d$, a \emph{red-white coloring} of $G$ is an assignment of red and white to the vertices in $G$ with at least one vertex being red. Under a red-white coloring, each vertex $v$ in $G$ has a \emph{code}, which is a $d$-vector, with the $i$-th coordinate being the number of red vertices that are of distance $i$ from $v$. A red-white coloring of $G$ where each vertex gets a distinct code is said to be an \emph{ID-coloring} of $G$. If $G$ has an ID-coloring, then it is an \emph{ID-graph}, and the \emph{ID-number} of $G$, denoted by $ID(G)$, is the minimum number of red vertices we need to construct an ID-coloring.

Kono and Zhang studied ID-trees in \cite{KZ1}; gave a note on ID-caterpillars in \cite{KZ2}; and studied ID-grids and ID-prisms in \cite{KZ3}. The results in these papers were also presented in Kono's dissertation \cite{Ko}.

In this paper, we introduce the \emph{ID-index} of a graph. We let $G=(V,\ E)$ be a graph with diameter $d$, let $f:V\longrightarrow \R$ assign real numbers to the vertices in $G$, and say $f(v)$ is the \emph{rank} of $v$ for $v\in V$. Under $f$, each $v\in V$ gets a \emph{string}, which is a $d$-vector with the $i$-th coordinate being the sum of the ranks of the vertices that are of distance $i$ from $v$. The \emph{ID-index} of $G$, denoted by $IDI(G)$, is defined to be the minimum number $k$ for which there is an $f:V\longrightarrow \R$ with $|f(V)|=k$, such that each vertex $v\in V$ gets a distinct string under $f$.

For example, the ID-index of the Petersen graph is three. In Figure \ref{petersen}, we have a construction showing that three ranks are enough. Note that, in the figure, the label on each vertex is in the form "rank, (string)", so "$3,\ (2,\ 11)$" means the rank of this vertex is $3$, and the string of this vertex is $(2,\ 11)$. Later we will explain why the ID-index of the Petersen graph cannot be one or two.

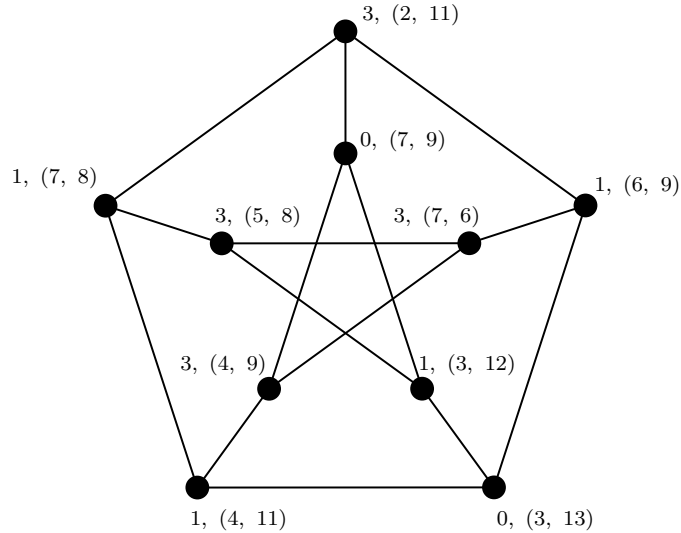
\begin{figure}[H]
    \tikzset{every picture/.style={line width=0.75pt}} 

\begin{tikzpicture}[x=0.75pt,y=0.75pt,yscale=-1,xscale=1]

\draw   (396.91,140.99) -- (350.76,283.29) -- (201.16,283.37) -- (154.85,141.11) -- (275.84,53.12) -- cycle ;
\draw    (275.84,53.12) -- (275.87,114.72) ;
\draw    (154.85,141.11) -- (213.45,160.12) ;
\draw    (201.16,283.37) -- (237.34,233.51) ;
\draw    (350.76,283.29) -- (314.52,233.47) ;
\draw    (396.91,140.99) -- (338.33,160.05) ;
\draw  [fill={rgb, 255:red, 0; green, 0; blue, 0 }  ,fill opacity=1 ] (270.34,53.12) .. controls (270.34,50.08) and (272.8,47.62) .. (275.84,47.62) .. controls (278.87,47.62) and (281.34,50.08) .. (281.34,53.12) .. controls (281.34,56.16) and (278.87,58.62) .. (275.84,58.62) .. controls (272.8,58.62) and (270.34,56.16) .. (270.34,53.12) -- cycle ;
\draw  [fill={rgb, 255:red, 0; green, 0; blue, 0 }  ,fill opacity=1 ] (270.37,114.72) .. controls (270.37,111.68) and (272.83,109.22) .. (275.87,109.22) .. controls (278.91,109.22) and (281.37,111.68) .. (281.37,114.72) .. controls (281.37,117.76) and (278.91,120.22) .. (275.87,120.22) .. controls (272.83,120.22) and (270.37,117.76) .. (270.37,114.72) -- cycle ;
\draw  [fill={rgb, 255:red, 0; green, 0; blue, 0 }  ,fill opacity=1 ] (149.35,141.11) .. controls (149.35,138.08) and (151.82,135.61) .. (154.85,135.61) .. controls (157.89,135.61) and (160.35,138.08) .. (160.35,141.11) .. controls (160.35,144.15) and (157.89,146.61) .. (154.85,146.61) .. controls (151.82,146.61) and (149.35,144.15) .. (149.35,141.11) -- cycle ;
\draw  [fill={rgb, 255:red, 0; green, 0; blue, 0 }  ,fill opacity=1 ] (207.95,160.12) .. controls (207.95,157.08) and (210.41,154.62) .. (213.45,154.62) .. controls (216.49,154.62) and (218.95,157.08) .. (218.95,160.12) .. controls (218.95,163.16) and (216.49,165.62) .. (213.45,165.62) .. controls (210.41,165.62) and (207.95,163.16) .. (207.95,160.12) -- cycle ;
\draw  [fill={rgb, 255:red, 0; green, 0; blue, 0 }  ,fill opacity=1 ] (195.66,283.37) .. controls (195.66,280.33) and (198.12,277.87) .. (201.16,277.87) .. controls (204.19,277.87) and (206.66,280.33) .. (206.66,283.37) .. controls (206.66,286.4) and (204.19,288.87) .. (201.16,288.87) .. controls (198.12,288.87) and (195.66,286.4) .. (195.66,283.37) -- cycle ;
\draw  [fill={rgb, 255:red, 0; green, 0; blue, 0 }  ,fill opacity=1 ] (231.84,233.51) .. controls (231.84,230.47) and (234.3,228.01) .. (237.34,228.01) .. controls (240.38,228.01) and (242.84,230.47) .. (242.84,233.51) .. controls (242.84,236.55) and (240.38,239.01) .. (237.34,239.01) .. controls (234.3,239.01) and (231.84,236.55) .. (231.84,233.51) -- cycle ;
\draw  [fill={rgb, 255:red, 0; green, 0; blue, 0 }  ,fill opacity=1 ] (345.26,283.29) .. controls (345.26,280.25) and (347.72,277.79) .. (350.76,277.79) .. controls (353.79,277.79) and (356.26,280.25) .. (356.26,283.29) .. controls (356.26,286.33) and (353.79,288.79) .. (350.76,288.79) .. controls (347.72,288.79) and (345.26,286.33) .. (345.26,283.29) -- cycle ;
\draw  [fill={rgb, 255:red, 0; green, 0; blue, 0 }  ,fill opacity=1 ] (309.02,233.47) .. controls (309.02,230.43) and (311.48,227.97) .. (314.52,227.97) .. controls (317.56,227.97) and (320.02,230.43) .. (320.02,233.47) .. controls (320.02,236.51) and (317.56,238.97) .. (314.52,238.97) .. controls (311.48,238.97) and (309.02,236.51) .. (309.02,233.47) -- cycle ;
\draw  [fill={rgb, 255:red, 0; green, 0; blue, 0 }  ,fill opacity=1 ] (391.41,140.99) .. controls (391.41,137.95) and (393.87,135.49) .. (396.91,135.49) .. controls (399.95,135.49) and (402.41,137.95) .. (402.41,140.99) .. controls (402.41,144.02) and (399.95,146.49) .. (396.91,146.49) .. controls (393.87,146.49) and (391.41,144.02) .. (391.41,140.99) -- cycle ;
\draw  [fill={rgb, 255:red, 0; green, 0; blue, 0 }  ,fill opacity=1 ] (332.83,160.05) .. controls (332.83,157.02) and (335.29,154.55) .. (338.33,154.55) .. controls (341.37,154.55) and (343.83,157.02) .. (343.83,160.05) .. controls (343.83,163.09) and (341.37,165.55) .. (338.33,165.55) .. controls (335.29,165.55) and (332.83,163.09) .. (332.83,160.05) -- cycle ;
\draw    (213.45,160.12) -- (338.33,160.05) ;
\draw    (338.33,160.05) -- (237.34,233.51) ;
\draw    (237.34,233.51) -- (275.87,114.72) ;
\draw    (275.87,114.72) -- (314.52,233.47) ;
\draw    (213.45,160.12) -- (314.52,233.47) ;

\draw (283,38) node [anchor=north west][inner sep=0.75pt]  [font=\scriptsize] [align=left] {$\displaystyle 3,\ ( 2,\ 11)$};
\draw (106,120) node [anchor=north west][inner sep=0.75pt]  [font=\scriptsize] [align=left] {$\displaystyle 1,\ ( 7,\ 8)$};
\draw (196,293) node [anchor=north west][inner sep=0.75pt]  [font=\scriptsize] [align=left] {$\displaystyle 1,\ ( 4,\ 11)$};
\draw (350.76,292.79) node [anchor=north west][inner sep=0.75pt]  [font=\scriptsize] [align=left] {$\displaystyle 0,\ ( 3,\ 13)$};
\draw (400,124) node [anchor=north west][inner sep=0.75pt]  [font=\scriptsize] [align=left] {$\displaystyle 1,\ ( 6,\ 9)$};
\draw (281.72,101.66) node [anchor=north west][inner sep=0.75pt]  [font=\scriptsize] [align=left] {$\displaystyle 0,\ ( 7,\ 9)$};
\draw (208.68,139.93) node [anchor=north west][inner sep=0.75pt]  [font=\scriptsize] [align=left] {$\displaystyle 3,\ ( 5,\ 8)$};
\draw (191.18,214.43) node [anchor=north west][inner sep=0.75pt]  [font=\scriptsize] [align=left] {$\displaystyle 3,\ ( 4,\ 9)$};
\draw (311.18,214.43) node [anchor=north west][inner sep=0.75pt]  [font=\scriptsize] [align=left] {$\displaystyle 1,\ ( 3,\ 12)$};
\draw (298.83,140.05) node [anchor=north west][inner sep=0.75pt]  [font=\scriptsize] [align=left] {$\displaystyle 3,\ ( 7,\ 6)$};

\end{tikzpicture}
\caption{Ranks and strings in the Petersen graph.}
\label{petersen}
\end{figure}

In fact, every graph $G=(V,\ E)$ has an ID-index. This is because we can assign $2^1,\ 2^2,\ ...,\ 2^{|V|}$ to the vertices in $G$, and this assignment always guarantees that each vertex gets a distinct string.

We have some connections between an ID-graph and its ID-index.

\begin{proposition} \label{prop1}
    If $G$ is an ID-graph, then $IDI(G)\le 2$.
\end{proposition}

\begin{proof}
    Assume $G$ is an ID-graph with an ID-coloring. We know that each vertex has a distinct code in this ID-coloring. Then we construct a rank assignment by letting the rank of each red vertex be one, and letting the rank of each white vertex be zero, then it is easy to see that the string of a vertex is the same as the code of this vertex, so each vertex has a distinct string, and thus $IDI(G)\le 2$.
\end{proof}

\begin{proposition} \label{prop2}
    For a graph $G$, the following two statements are equivalent.
    \begin{itemize}
        \item $G$ is an ID-graph, and there is an ID-coloring of $G$ where every vertex is red.
        \item $IDI(G)=1$.
    \end{itemize}
\end{proposition}

\begin{proof}
    If $G$ is an ID-graph with an ID-coloring where every vertex is red, then we can let the rank of every vertex be one, and each vertex will have a distinct string, which is the same as its code, thus $IDI(G)=1$. 
    
    For the other direction, if $IDI(G)=1$, then there is an $r\in \R$ with $r\neq 0$ such that we can let the rank of every vertex be $r$, and each vertex $v$ will get a distinct string $(c_{v1},\ c_{v2},\ ...,\ c_{vd})$, where $d$ is the diameter of $G$. Now if we let every vertex in $G$ be red, then we can see that the code of $v$ is $(\frac{c_{v1}}{r},\ \frac{c_{v2}}{r},\ ...,\ \frac{c_{vd}}{r})$, which means each vertex has a distinct code, and thus $G$ is an ID-graph with an ID-coloring where every vertex is red.
\end{proof}

The converse of Proposition \ref{prop1} is not true, because a graph $G$ with $IDI(G)=2$ is not necessarily an ID-graph. For example, the construction in Figure \ref{K112} shows that the ID-index of the complete tripartite graph $K_{1,\ 1,\ 2}$ is two (obviously it is not one). However, $K_{1,\ 1,\ 2}$ is not an ID-graph, because it is showed in \cite{CKZ} that the only ID-complete multipartite graphs are $K_{1,\ 1}$ and $K_{1,\ 2}$.

\begin{figure}[H]
    \tikzset{every picture/.style={line width=0.75pt}} 

\begin{tikzpicture}[x=0.75pt,y=0.75pt,yscale=-1,xscale=1]

\draw  [fill={rgb, 255:red, 0; green, 0; blue, 0 }  ,fill opacity=1 ] (82.5,70) .. controls (82.5,66.96) and (84.96,64.5) .. (88,64.5) .. controls (91.04,64.5) and (93.5,66.96) .. (93.5,70) .. controls (93.5,73.04) and (91.04,75.5) .. (88,75.5) .. controls (84.96,75.5) and (82.5,73.04) .. (82.5,70) -- cycle ;
\draw  [fill={rgb, 255:red, 0; green, 0; blue, 0 }  ,fill opacity=1 ] (181.5,70) .. controls (181.5,66.96) and (183.96,64.5) .. (187,64.5) .. controls (190.04,64.5) and (192.5,66.96) .. (192.5,70) .. controls (192.5,73.04) and (190.04,75.5) .. (187,75.5) .. controls (183.96,75.5) and (181.5,73.04) .. (181.5,70) -- cycle ;
\draw    (88,70) -- (187,70) ;
\draw  [fill={rgb, 255:red, 0; green, 0; blue, 0 }  ,fill opacity=1 ] (110.5,136) .. controls (110.5,132.96) and (112.96,130.5) .. (116,130.5) .. controls (119.04,130.5) and (121.5,132.96) .. (121.5,136) .. controls (121.5,139.04) and (119.04,141.5) .. (116,141.5) .. controls (112.96,141.5) and (110.5,139.04) .. (110.5,136) -- cycle ;
\draw  [fill={rgb, 255:red, 0; green, 0; blue, 0 }  ,fill opacity=1 ] (154.5,136) .. controls (154.5,132.96) and (156.96,130.5) .. (160,130.5) .. controls (163.04,130.5) and (165.5,132.96) .. (165.5,136) .. controls (165.5,139.04) and (163.04,141.5) .. (160,141.5) .. controls (156.96,141.5) and (154.5,139.04) .. (154.5,136) -- cycle ;
\draw    (88,70) -- (116,136) ;
\draw    (88,70) -- (160,136) ;
\draw    (187,70) -- (116,136) ;
\draw    (187,70) -- (160,136) ;

\draw (54,46) node [anchor=north west][inner sep=0.75pt]   [align=left] {$\displaystyle 1,\ ( 5,\ 0)$};
\draw (157,46) node [anchor=north west][inner sep=0.75pt]   [align=left] {$\displaystyle 2,\ ( 4,\ 0)$};
\draw (62,144) node [anchor=north west][inner sep=0.75pt]   [align=left] {$\displaystyle 1,\ ( 3,\ 2)$};
\draw (153,144) node [anchor=north west][inner sep=0.75pt]   [align=left] {$\displaystyle 2,\ ( 3,\ 1)$};

\end{tikzpicture}
\caption{Tripartitle graph $K_{1,\ 1,\ 2}$.}
\label{K112}
\end{figure}
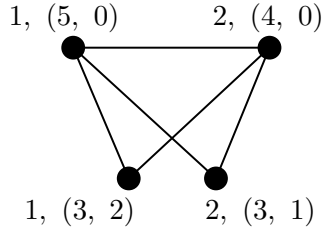

Also, this example tells us that, for a graph $G$ and its (induced) subgraph $G'$, we do not necessarily have $IDI(G')\le IDI(G)$. We can see that the triangle $K_3$ is an (induced) subgraph of $K_{1,\ 1,\ 2}$, and $IDI(K_3)=3>2=IDI(K_{1,\ 1,\ 2})$.

\section{ID-indices of different graphs}
In this section, we determine the ID-indices of some graphs, including paths, grids, cycles, prisms, complete graphs, some complete multipartite graphs, and some caterpillars.

\subsection{A lower bound}
Firstly we give a lower bound on the ID-index of a graph, using the $t$-tuplets defined in \cite{CKZ}. For a graph $G=(V,\ E)$, a set of vertices $S\subseteq V$ with $|S|=t\ge 2$ is said to be a $t$-tuplet if either
\begin{itemize}
    \item $S$ is an independent set, and any two vertices in $S$ have the same neighborhood;
\end{itemize}
or
\begin{itemize}
    \item $S$ is a clique, and any two vertices in $S$ have the same closed neighborhood, where the closed neighborhood of $v$ is just $N(v)\cup \{v\}$.
\end{itemize}

It is easy to see that, if $u$ and $v$ are two vertices in the same $t$-tuplet, then for any vertex $w$ other than $u$ and $v$, we have $d(u,\ w)=d(v,\ w)$, where $d(\cdot,\ \cdot)$ measures the distance between two vertices.

\begin{theorem} \label{lowerbound}
    Let $G=(V,\ E)$ be a graph, and let $T(G)$ be the largest $t$ such that there is a $t$-tuplet in $G$. We have
    \begin{align*}
        IDI(G)\ge T(G).
    \end{align*}
\end{theorem}

\begin{proof}
    Let $u$ and $v$ be two vertices in a $T(G)$-tuplet, and let $f:V\longrightarrow \R$ be a rank assignment under which we have $f(u)=f(v)=r$ for some $r\in \R$.

    We have, of course, $d(u,\ v)=d(v,\ u)=D$ for some $D$, so $u$ will contribute $r$ to the $D$-th coordinate in the string of $v$; and $v$ will also contribute $r$ to the $D$-th coordinate in the string of $u$. And for any vertex $w$ other than $u$ and $v$, as we just mentioned, we have $d(u,\ w)=d(v,\ w)$, so $w$ will contribute $f(w)$ to the same coordinate in the string of $u$ and in the string of $v$. So, under rank assignment $f$, $u$ and $v$ will get exactly the same string. Thus, in order to let each vertex have a distinct string, we need to let the vertices in a $T(G)$-tuplet get distinct ranks, which means $IDI(G)\ge T(G)$.
\end{proof}

Let us look at specific graphs.

\subsection{Some basic graphs}
In this subsection, we make usage of the results in \cite{CKZ,KZ3} to study some basic families of graphs, whose ID-indices are constants.

For paths, it is proved in \cite{CKZ} that every path $P_n$ with $n\ge 2$ is an ID-graph, so by Proposition \ref{prop1}, we have $IDI(P_n)\le 2$. And it is easy to see that every vertex being red does not give us an ID-coloring of $P_n$, so by Proposition \ref{prop2}, we know $IDI(P_n)\neq 1$. So for any path $P_n$ with $n\ge 2$, we have $IDI(P_n)=2$.

For grids, it is proved in \cite{KZ3} that the grid $P_m\square P_n$ is an ID-graph if and only if $(m,\ n)\neq (2,\ 2)$. Similar to paths, we can draw the conclusion that $IDI(P_m\square P_n)=2$ if $(m,\ n)\neq (2,\ 2)$. Also, we can make the observation that $IDI(P_2\square P_2)=IDI(C_4)=3$.

For cycles, it is proved in \cite{CKZ} that the cycle $C_n$ is an ID-graph if and only if $n\ge 6$. Similar to paths, we can reach the conclusion that $IDI(C_n)=2$ if $n\ge 6$. And we can make the observation that $IDI(C_3)=IDI(C_4)=IDI(C_5)=3$.

For prisms, it is proved in \cite{KZ3} that the prism $Y_n:=C_n\square P_2$ is an ID-graph if and only if $n\ge 6$. Similar to paths, we can get the conclusion that $IDI(Y_n)=2$ if $n\ge 6$. Then, for $Y_3$, $Y_4$, and $Y_5$, by construction, we know that the ID-index of each of them is at most three. Figure \ref{Y5} shows a feasible rank assignment and the strings under this assignment for $Y_5$, and the constructions for $Y_3$ and $Y_4$ are very similar.

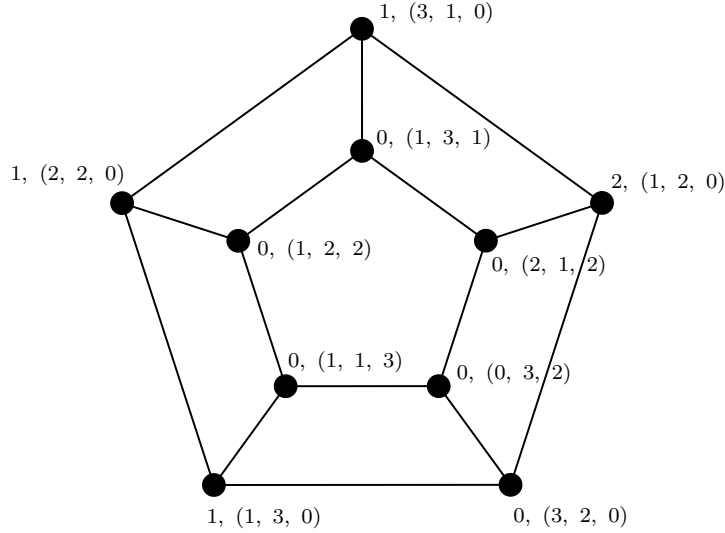
\begin{figure}[H]
    \tikzset{every picture/.style={line width=0.75pt}} 

\begin{tikzpicture}[x=0.75pt,y=0.75pt,yscale=-1,xscale=1]

\draw   (453.91,361.99) -- (407.76,504.29) -- (258.16,504.37) -- (211.85,362.11) -- (332.84,274.12) -- cycle ;
\draw   (395.33,381.05) -- (371.52,454.47) -- (294.34,454.51) -- (270.45,381.12) -- (332.87,335.72) -- cycle ;
\draw    (332.84,274.12) -- (332.87,335.72) ;
\draw    (211.85,362.11) -- (270.45,381.12) ;
\draw    (258.16,504.37) -- (294.34,454.51) ;
\draw    (407.76,504.29) -- (371.52,454.47) ;
\draw    (453.91,361.99) -- (395.33,381.05) ;
\draw  [fill={rgb, 255:red, 0; green, 0; blue, 0 }  ,fill opacity=1 ] (327.34,274.12) .. controls (327.34,271.08) and (329.8,268.62) .. (332.84,268.62) .. controls (335.87,268.62) and (338.34,271.08) .. (338.34,274.12) .. controls (338.34,277.16) and (335.87,279.62) .. (332.84,279.62) .. controls (329.8,279.62) and (327.34,277.16) .. (327.34,274.12) -- cycle ;
\draw  [fill={rgb, 255:red, 0; green, 0; blue, 0 }  ,fill opacity=1 ] (327.37,335.72) .. controls (327.37,332.68) and (329.83,330.22) .. (332.87,330.22) .. controls (335.91,330.22) and (338.37,332.68) .. (338.37,335.72) .. controls (338.37,338.76) and (335.91,341.22) .. (332.87,341.22) .. controls (329.83,341.22) and (327.37,338.76) .. (327.37,335.72) -- cycle ;
\draw  [fill={rgb, 255:red, 0; green, 0; blue, 0 }  ,fill opacity=1 ] (206.35,362.11) .. controls (206.35,359.08) and (208.82,356.61) .. (211.85,356.61) .. controls (214.89,356.61) and (217.35,359.08) .. (217.35,362.11) .. controls (217.35,365.15) and (214.89,367.61) .. (211.85,367.61) .. controls (208.82,367.61) and (206.35,365.15) .. (206.35,362.11) -- cycle ;
\draw  [fill={rgb, 255:red, 0; green, 0; blue, 0 }  ,fill opacity=1 ] (264.95,381.12) .. controls (264.95,378.08) and (267.41,375.62) .. (270.45,375.62) .. controls (273.49,375.62) and (275.95,378.08) .. (275.95,381.12) .. controls (275.95,384.16) and (273.49,386.62) .. (270.45,386.62) .. controls (267.41,386.62) and (264.95,384.16) .. (264.95,381.12) -- cycle ;
\draw  [fill={rgb, 255:red, 0; green, 0; blue, 0 }  ,fill opacity=1 ] (252.66,504.37) .. controls (252.66,501.33) and (255.12,498.87) .. (258.16,498.87) .. controls (261.19,498.87) and (263.66,501.33) .. (263.66,504.37) .. controls (263.66,507.4) and (261.19,509.87) .. (258.16,509.87) .. controls (255.12,509.87) and (252.66,507.4) .. (252.66,504.37) -- cycle ;
\draw  [fill={rgb, 255:red, 0; green, 0; blue, 0 }  ,fill opacity=1 ] (288.84,454.51) .. controls (288.84,451.47) and (291.3,449.01) .. (294.34,449.01) .. controls (297.38,449.01) and (299.84,451.47) .. (299.84,454.51) .. controls (299.84,457.55) and (297.38,460.01) .. (294.34,460.01) .. controls (291.3,460.01) and (288.84,457.55) .. (288.84,454.51) -- cycle ;
\draw  [fill={rgb, 255:red, 0; green, 0; blue, 0 }  ,fill opacity=1 ] (402.26,504.29) .. controls (402.26,501.25) and (404.72,498.79) .. (407.76,498.79) .. controls (410.79,498.79) and (413.26,501.25) .. (413.26,504.29) .. controls (413.26,507.33) and (410.79,509.79) .. (407.76,509.79) .. controls (404.72,509.79) and (402.26,507.33) .. (402.26,504.29) -- cycle ;
\draw  [fill={rgb, 255:red, 0; green, 0; blue, 0 }  ,fill opacity=1 ] (366.02,454.47) .. controls (366.02,451.43) and (368.48,448.97) .. (371.52,448.97) .. controls (374.56,448.97) and (377.02,451.43) .. (377.02,454.47) .. controls (377.02,457.51) and (374.56,459.97) .. (371.52,459.97) .. controls (368.48,459.97) and (366.02,457.51) .. (366.02,454.47) -- cycle ;
\draw  [fill={rgb, 255:red, 0; green, 0; blue, 0 }  ,fill opacity=1 ] (448.41,361.99) .. controls (448.41,358.95) and (450.87,356.49) .. (453.91,356.49) .. controls (456.95,356.49) and (459.41,358.95) .. (459.41,361.99) .. controls (459.41,365.02) and (456.95,367.49) .. (453.91,367.49) .. controls (450.87,367.49) and (448.41,365.02) .. (448.41,361.99) -- cycle ;
\draw  [fill={rgb, 255:red, 0; green, 0; blue, 0 }  ,fill opacity=1 ] (389.83,381.05) .. controls (389.83,378.02) and (392.29,375.55) .. (395.33,375.55) .. controls (398.37,375.55) and (400.83,378.02) .. (400.83,381.05) .. controls (400.83,384.09) and (398.37,386.55) .. (395.33,386.55) .. controls (392.29,386.55) and (389.83,384.09) .. (389.83,381.05) -- cycle ;

\draw (340,259) node [anchor=north west][inner sep=0.75pt]  [font=\scriptsize] [align=left] {$\displaystyle 1,\ ( 3,\ 1,\ 0)$};
\draw (154,341) node [anchor=north west][inner sep=0.75pt]  [font=\scriptsize] [align=left] {$\displaystyle 1,\ ( 2,\ 2,\ 0)$};
\draw (253,514) node [anchor=north west][inner sep=0.75pt]  [font=\scriptsize] [align=left] {$\displaystyle 1,\ ( 1,\ 3,\ 0)$};
\draw (407.76,513.79) node [anchor=north west][inner sep=0.75pt]  [font=\scriptsize] [align=left] {$\displaystyle 0,\ ( 3,\ 2,\ 0)$};
\draw (457,345) node [anchor=north west][inner sep=0.75pt]  [font=\scriptsize] [align=left] {$\displaystyle 2,\ ( 1,\ 2,\ 0)$};
\draw (338.72,322.66) node [anchor=north west][inner sep=0.75pt]  [font=\scriptsize] [align=left] {$\displaystyle 0,\ ( 1,\ 3,\ 1)$};
\draw (278.68,378.93) node [anchor=north west][inner sep=0.75pt]  [font=\scriptsize] [align=left] {$\displaystyle 0,\ ( 1,\ 2,\ 2)$};
\draw (294.18,435.43) node [anchor=north west][inner sep=0.75pt]  [font=\scriptsize] [align=left] {$\displaystyle 0,\ ( 1,\ 1,\ 3)$};
\draw (379.18,441.43) node [anchor=north west][inner sep=0.75pt]  [font=\scriptsize] [align=left] {$\displaystyle 0,\ ( 0,\ 3,\ 2)$};
\draw (396.83,387.05) node [anchor=north west][inner sep=0.75pt]  [font=\scriptsize] [align=left] {$\displaystyle 0,\ ( 2,\ 1,\ 2)$};

\end{tikzpicture}
\caption{Ranks and strings in $Y_5$.}
\label{Y5}
\end{figure}

Actually we have $IDI(Y_3)=IDI(Y_4)=IDI(Y_5)=3$. To verify this, we just need to show that each of them is not two, because obviously they cannot be one. Of course, one may show this by a case-by-case argument, but we can also do it in a much easier way.

We mentioned that, in general, a graph $G$ with $IDI(G)=2$ is not necessarily an ID-graph. However, if $G$ has a specific property stated in the following lemma, then it is an ID-graph.

\begin{lemma} \label{lemma}
    Let $G=(V,\ E)$ be a graph with diameter $d$. If there exist $n_1,\ n_2,\ ...,\ n_d\ge 1$ with $\sum_{i=1}^d n_i=|V|-1$ such that for any vertex $v\in V$ and any $1\le i\le d$, there are $n_i$ vertices that are of distance $i$ from $v$, then for any rank assignment $f$ which gives each vertex a distinct string, the rank assignment $kf+b:v\longmapsto kf(v)+b$ with $k\neq 0$ also gives each vertex a distinct string. Furthermore, if we also have $IDI(G)=2$, then $G$ is an ID-graph.
\end{lemma}

\begin{proof}
    Assume $G$ is a graph with such $n_1,\ n_2,\ ...,\ n_d$. For vertices $u$ and $v$, assume they have strings $(c_{u1},\ c_{u2},\ ...,\ c_{ud})$ and $(c_{v1},\ c_{v2},\ ...,\ c_{vd})$ under rank assignment $f$, and we have $c_{ui}\neq c_{vi}$ for some $1\le i\le d$. Then under $kf+b$ with $k\neq 0$, the $i$-th coordinate in the string of $u$ will be $kc_{ui}+n_i b$, and the $i$-th coordinate in the string of $v$ will be $kc_{vi}+n_i b$. We have $kc_{ui}+n_i b\neq kc_{vi}+n_i b$ because $c_{ui}\neq c_{vi}$ and $k\neq 0$. So $kf+b$ also gives each vertex a distinct string.

    Now we further assume $IDI(G)=2$, so there is some $f$ which assigns $r_1\in \R$ to some vertices, and assigns $r_2$ with $r_2\neq r_1$ to others, such that each vertex has a distinct string. Then it is easy to see that there exist $k\neq 0$ and $b$ such that $kr_1+b=0$ and $kr_2+b=1$. And we know $kf+b$ is a rank assignment where each vertex has a distinct string. But now, we can let every vertex with rank zero be white, and let every vertex with rank one be red, then we will get an ID-coloring of $G$ where the code of a vertex is the same as the string of this vertex, so $G$ is an ID-graph.
\end{proof}

\begin{proposition}
    We have $IDI(Y_n)\neq 2$ for $n\in \{3,\ 4,\ 5\}$.
\end{proposition}

\begin{proof}
    We can apply Lemma \ref{lemma} to $Y_3$, $Y_4$, and $Y_5$, because:
    \begin{itemize}
        \item For $Y_3$, we have: $diam(Y_3)=2$, $n_1=3$, and $n_2=2$.
        \item For $Y_4$, we have: $diam(Y_4)=3$, $n_1=3$, $n_2=3$, and $n_3=1$.
        \item For $Y_5$, we have: $diam(Y_5)=3$, $n_1=3$, $n_2=4$, and $n_3=2$.
    \end{itemize}
    
    By Lemma \ref{lemma}, we know that if $IDI(Y_n)=2$ for some $n\in \{3,\ 4,\ 5\}$, then $Y_n$ is an ID-graph. However, as we mentioned earlier, it is proved in \cite{KZ3} that $Y_n$ is an ID-graph if and only if $n\ge 6$. So $IDI(Y_n)\neq 2$ for $n\in \{3,\ 4,\ 5\}$.
\end{proof}

Note that Lemma \ref{lemma} also applies to the Petersen graph, so if the ID-index of the Petersen graph is two, then the Petersen graph is an ID-graph. However, it is showed in \cite{CKZ} that the Petersen graph is not an ID-graph, so its ID-index cannot be two. Combining this fact with the observation that its ID-index is not one, and the construction in Figure \ref{petersen}, we know that indeed the ID-index of the Petersen graph is three.

\subsection{Complete graphs and complete multipartite graphs}
In this subsection, the lower bound in Theorem \ref{lowerbound} plays a key role.

Firstly, the complete graph $K_n$ itself is an $n$-tuplet, so by Theorem \ref{lowerbound}, we know $IDI(K_n)\ge n$, and thus $IDI(K_n)=n$.

Then for the complete multipartitle graphs, we can see that the vertices in each part form a $t$-tuplet with $t$ being the part size. So, in the complete $k$-partite graph $K_{m_1,\ m_2,\ ...,\ m_k}$ with $m_1\le m_2\le ...\le m_k$, we know that $IDI(K_{m_1,\ m_2,\ ...,\ m_k})\ge m_k$.

We can determine the exact ID-indices for some of the complete multipartitle graphs.

\begin{theorem}
    For the complete $k$-partite graph $K_{m_1,\ m_2,\ ...,\ m_k}$ with $m_1<m_2<...<m_k$, we have
    \begin{align*}
        IDI(K_{m_1,\ m_2,\ ...,\ m_k})=m_k.
    \end{align*}
\end{theorem}

\begin{proof}
    We already know $IDI(K_{m_1,\ m_2,\ ...,\ m_k})\ge m_k$, and we will show $IDI(K_{m_1,\ m_2,\ ...,\ m_k})\le m_k$ by construction. We construct a rank assignment $f$ in the following fashion: For each $1\le i\le k$, there are $m_i$ vertices in the $i$-th part, and we let the ranks of these vertices be $1,\ 2,\ ...,\ m_i$. We denote $\sum_{i=1}^k\sum_{j=1}^{m_i}j$ by $M$, so $M$ is just the sum of the ranks of all vertices.

    Under this $f$, two vertices in different parts have different strings, because their first coordinates must vary. Assuming $u$ is in the $p$-th part, and $v$ is in the $q$-th part, with $p\neq q$, we have that the first coordinate in the string of $u$ is $M-\sum_{\ell=1}^{m_p}\ell$, the first coordinate in the string of $v$ is $M-\sum_{l=1}^{m_q}l$, and they two are different because $m_p\neq m_q$.

    Also, two vertices in the same part have different strings, because their second coordinates must vary. Assuming $u$ and $v$ are both in the $r$-th part, with $1\le f(u)\neq f(v)\le m_r$, we have that the second coordinate in the string of $u$ is $\sum_{\ell=1}^{m_r}-f(u)$, the second coordinate in the string of $v$ is $\sum_{\ell=1}^{m_r}-f(v)$, and they two are different because $f(u)\neq f(v)$.

    Under this $f$, each vertex has a distinct string, so $IDI(K_{m_1,\ m_2,\ ...,\ m_k})\le m_k$, and thus $IDI(K_{m_1,\ m_2,\ ...,\ m_k})=m_k$.
\end{proof}

Also, we know the exact ID-indices of the complete bipartite graphs.

\begin{theorem} \label{bipartite}
    For the complete bipartite graph $K_{m,\ n}$ with $m\le n$, we have
    \[IDI(K_{m,\ n})=\begin{cases}
        n &if\ m<n, \\
        n+1 &if\ m=n.
    \end{cases}\]
\end{theorem}

\begin{proof}
    The case that $m<n$ is already solved in the previous theorem. Let us assume $m=n$.

    Firstly, we show that $IDI(K_{n,\ n})>n$. Assume we only have $n$ ranks $r_1,\ r_2,\ ...,\ r_n$. Because each part is an $n$-tuplet, we know that, in each part, the $n$ ranks must all appear. But now, it is easy to see that, for any $1\le i\le n$, the vertex in the first part with rank $r_i$ and the vertex in the second part with rank $r_i$ have the same string, which we do not want to happen. So $IDI(K_{n,\ n})>n$.

    Then, we show that $IDI(K_{n,\ n})\le n+1$ by construction. We construct $f$ by letting the vertices in the first part get ranks $1,\ 2,\ ...,\ n$, and letting the vertices in the second part get ranks $2,\ 3,\ ...,\ n+1$. Then, the same as what we did in the previous proof, it is easy to check that any two vertices have different strings. 
\end{proof}

Following the same idea, we may also figure something out for the complete $k$-partite graphs with $k$ being small. For general complete multipartite graphs, it is easy to see the following lower bound.

\begin{proposition}
    For $K_{n_1\times 1, n_2\times 2, ..., n_t\times t}$, the complete $(n_1+n_2+...+n_t)$-partite graph with $n_i$ parts having size $i$ for $1\le i\le t$, we have
    \begin{align*}
        {IDI(K_{n_1\times 1, n_2\times 2, ..., n_t\times t})\choose i}\ge n_i
    \end{align*}
    for any $1\le i\le t$.
\end{proposition}

But new ideas shall be needed to determine the exact ID-indices for general complete multipartite graphs.

\begin{problem}
    Determine the exact value of $IDI(K_{n_1\times 1, n_2\times 2, ..., n_t\times t})$.
\end{problem}

\subsection{Caterpillars}
The last part of this paper is devoted to finding the ID-indices of caterpillars. A \emph{caterpillar} is a tree, where we have a path, called the \emph{spine}, and some leaves attached to the spine vertices. In a caterpillar, the degree of a vertex is at least two if and only if this vertex is on the spine.

Let $G$ be a caterpillar with the spine being a path on $n$ vertices. From one end of the spine to the other, we denote the spine vertices by $s_1,\ s_2,\ ...,\ s_n$, and denote the number of leaves attached on $s_i$ by $L_i$. For example, Figure \ref{caterpillar} shows a caterpillar with $L_1=2$, $L_2=3$, $L_3=2$, $L_4=0$, and $L_5=3$.

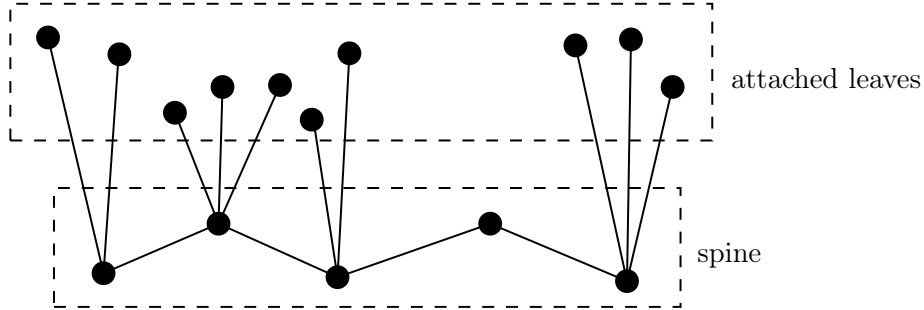
\begin{figure}[H]
    \tikzset{every picture/.style={line width=0.75pt}} 

\begin{tikzpicture}[x=0.75pt,y=0.75pt,yscale=-1,xscale=1]

\draw    (141,165) -- (113,46) ;
\draw    (141,165) -- (149,49) ;
\draw    (177,84) -- (199,140) ;
\draw    (379,50) -- (405,169) ;
\draw    (201,71) -- (199,140) ;
\draw    (199,140) -- (230,70) ;
\draw    (259,167) -- (246,82) ;
\draw    (259,167) -- (265,54) ;
\draw    (405,169) -- (407,47) ;
\draw    (405,169) -- (428,71) ;
\draw  [fill={rgb, 255:red, 0; green, 0; blue, 0 }  ,fill opacity=1 ] (107.5,46) .. controls (107.5,42.96) and (109.96,40.5) .. (113,40.5) .. controls (116.04,40.5) and (118.5,42.96) .. (118.5,46) .. controls (118.5,49.04) and (116.04,51.5) .. (113,51.5) .. controls (109.96,51.5) and (107.5,49.04) .. (107.5,46) -- cycle ;
\draw  [fill={rgb, 255:red, 0; green, 0; blue, 0 }  ,fill opacity=1 ] (143.5,54.5) .. controls (143.5,51.46) and (145.96,49) .. (149,49) .. controls (152.04,49) and (154.5,51.46) .. (154.5,54.5) .. controls (154.5,57.54) and (152.04,60) .. (149,60) .. controls (145.96,60) and (143.5,57.54) .. (143.5,54.5) -- cycle ;
\draw  [fill={rgb, 255:red, 0; green, 0; blue, 0 }  ,fill opacity=1 ] (171.5,84) .. controls (171.5,80.96) and (173.96,78.5) .. (177,78.5) .. controls (180.04,78.5) and (182.5,80.96) .. (182.5,84) .. controls (182.5,87.04) and (180.04,89.5) .. (177,89.5) .. controls (173.96,89.5) and (171.5,87.04) .. (171.5,84) -- cycle ;
\draw  [fill={rgb, 255:red, 0; green, 0; blue, 0 }  ,fill opacity=1 ] (195.5,71) .. controls (195.5,67.96) and (197.96,65.5) .. (201,65.5) .. controls (204.04,65.5) and (206.5,67.96) .. (206.5,71) .. controls (206.5,74.04) and (204.04,76.5) .. (201,76.5) .. controls (197.96,76.5) and (195.5,74.04) .. (195.5,71) -- cycle ;
\draw  [fill={rgb, 255:red, 0; green, 0; blue, 0 }  ,fill opacity=1 ] (224.5,70) .. controls (224.5,66.96) and (226.96,64.5) .. (230,64.5) .. controls (233.04,64.5) and (235.5,66.96) .. (235.5,70) .. controls (235.5,73.04) and (233.04,75.5) .. (230,75.5) .. controls (226.96,75.5) and (224.5,73.04) .. (224.5,70) -- cycle ;
\draw  [fill={rgb, 255:red, 0; green, 0; blue, 0 }  ,fill opacity=1 ] (240.5,87.5) .. controls (240.5,84.46) and (242.96,82) .. (246,82) .. controls (249.04,82) and (251.5,84.46) .. (251.5,87.5) .. controls (251.5,90.54) and (249.04,93) .. (246,93) .. controls (242.96,93) and (240.5,90.54) .. (240.5,87.5) -- cycle ;
\draw  [fill={rgb, 255:red, 0; green, 0; blue, 0 }  ,fill opacity=1 ] (259.5,54) .. controls (259.5,50.96) and (261.96,48.5) .. (265,48.5) .. controls (268.04,48.5) and (270.5,50.96) .. (270.5,54) .. controls (270.5,57.04) and (268.04,59.5) .. (265,59.5) .. controls (261.96,59.5) and (259.5,57.04) .. (259.5,54) -- cycle ;
\draw  [fill={rgb, 255:red, 0; green, 0; blue, 0 }  ,fill opacity=1 ] (373.5,50) .. controls (373.5,46.96) and (375.96,44.5) .. (379,44.5) .. controls (382.04,44.5) and (384.5,46.96) .. (384.5,50) .. controls (384.5,53.04) and (382.04,55.5) .. (379,55.5) .. controls (375.96,55.5) and (373.5,53.04) .. (373.5,50) -- cycle ;
\draw  [fill={rgb, 255:red, 0; green, 0; blue, 0 }  ,fill opacity=1 ] (401.5,47) .. controls (401.5,43.96) and (403.96,41.5) .. (407,41.5) .. controls (410.04,41.5) and (412.5,43.96) .. (412.5,47) .. controls (412.5,50.04) and (410.04,52.5) .. (407,52.5) .. controls (403.96,52.5) and (401.5,50.04) .. (401.5,47) -- cycle ;
\draw  [fill={rgb, 255:red, 0; green, 0; blue, 0 }  ,fill opacity=1 ] (422.5,71) .. controls (422.5,67.96) and (424.96,65.5) .. (428,65.5) .. controls (431.04,65.5) and (433.5,67.96) .. (433.5,71) .. controls (433.5,74.04) and (431.04,76.5) .. (428,76.5) .. controls (424.96,76.5) and (422.5,74.04) .. (422.5,71) -- cycle ;
\draw  [fill={rgb, 255:red, 0; green, 0; blue, 0 }  ,fill opacity=1 ] (135.5,165) .. controls (135.5,161.96) and (137.96,159.5) .. (141,159.5) .. controls (144.04,159.5) and (146.5,161.96) .. (146.5,165) .. controls (146.5,168.04) and (144.04,170.5) .. (141,170.5) .. controls (137.96,170.5) and (135.5,168.04) .. (135.5,165) -- cycle ;
\draw  [fill={rgb, 255:red, 0; green, 0; blue, 0 }  ,fill opacity=1 ] (193.5,140) .. controls (193.5,136.96) and (195.96,134.5) .. (199,134.5) .. controls (202.04,134.5) and (204.5,136.96) .. (204.5,140) .. controls (204.5,143.04) and (202.04,145.5) .. (199,145.5) .. controls (195.96,145.5) and (193.5,143.04) .. (193.5,140) -- cycle ;
\draw  [fill={rgb, 255:red, 0; green, 0; blue, 0 }  ,fill opacity=1 ] (253.5,167) .. controls (253.5,163.96) and (255.96,161.5) .. (259,161.5) .. controls (262.04,161.5) and (264.5,163.96) .. (264.5,167) .. controls (264.5,170.04) and (262.04,172.5) .. (259,172.5) .. controls (255.96,172.5) and (253.5,170.04) .. (253.5,167) -- cycle ;
\draw  [fill={rgb, 255:red, 0; green, 0; blue, 0 }  ,fill opacity=1 ] (330.5,140) .. controls (330.5,136.96) and (332.96,134.5) .. (336,134.5) .. controls (339.04,134.5) and (341.5,136.96) .. (341.5,140) .. controls (341.5,143.04) and (339.04,145.5) .. (336,145.5) .. controls (332.96,145.5) and (330.5,143.04) .. (330.5,140) -- cycle ;
\draw  [fill={rgb, 255:red, 0; green, 0; blue, 0 }  ,fill opacity=1 ] (399.5,169) .. controls (399.5,165.96) and (401.96,163.5) .. (405,163.5) .. controls (408.04,163.5) and (410.5,165.96) .. (410.5,169) .. controls (410.5,172.04) and (408.04,174.5) .. (405,174.5) .. controls (401.96,174.5) and (399.5,172.04) .. (399.5,169) -- cycle ;
\draw  [dash pattern={on 4.5pt off 4.5pt}] (94,29) -- (448,29) -- (448,98) -- (94,98) -- cycle ;
\draw  [dash pattern={on 4.5pt off 4.5pt}] (116,122) -- (432,122) -- (432,182) -- (116,182) -- cycle ;
\draw    (141,165) -- (199,140) ;
\draw    (199,140) -- (259,167) ;
\draw    (259,167) -- (336,140) ;
\draw    (336,140) -- (405,169) ;

\draw (439,148) node [anchor=north west][inner sep=0.75pt]   [align=left] {spine};
\draw (456,60) node [anchor=north west][inner sep=0.75pt]   [align=left] {attached leaves};

\end{tikzpicture}
\caption{A caterpillar.}
\label{caterpillar}
\end{figure}

Note that $L_2,\ L_3,\ ...,\ L_{n-1}$ could be zero, but $L_1$ and $L_n$ must be at least one, because if a spine vertex on the end does not have any attached leaves, then this vertex will be considered a leaf attached to the spine vertex next to it.

As we can see, the diameter of $G$ is $n+1$; and for each spine vertex $s_i$, the $L_i$ leaves attached on it form an $L_i$-tuplet. So, by Theorem \ref{lowerbound}, we know $IDI(G)\ge \max_{1\le i\le n}L_i$. On the basis of this lower bound, we can determine the exact ID-indices of symmetric caterpillars, where a caterpillar is \emph{symmetric} if for any $1\le j\le \lfloor\frac{n}{2}\rfloor$, we have $L_j=L_{n+1-j}$.

\begin{theorem}
    Let $G=(V,\ E)$ be a symmetric caterpillar with $n$ spine vertices, and denote $\max_{1\le i\le n}L_i$ by $L$. We have
    \[IDI(G)=\begin{cases}
        L &if\ L\ge 2, \\
        2 &if\ L=1.
    \end{cases}\]
\end{theorem}

\begin{proof}
    If $n=1$, then the caterpillar is actually a star, which is just $K_{1,\ L}$. So we can get our desired conclusion using Theorem \ref{bipartite}. Now, let us assume there are at least two vertices on the spine, so $n\ge 2$.
    
    \textbf{Case 1.} $L\ge 2$.
    
    We already have $IDI(G)\ge L$, so we only need to check that $IDI(G)\le L$. We construct a rank assignment $f:V\longrightarrow \R$ in the following fashion: For the spine vertices, we let $f(s_1)=2$, and let $f(s_i)=1$ for $2\le i\le n$; for the $L_i$ leaves attached on $s_i$, we let their ranks be $1,\ 2,\ ...,\ L_i$. An example is given in Figure \ref{caterpillar2}, where the leftmost spine vertex is $s_1$, and the rightmost spine vertex is $s_6$.

    \begin{figure}[H]
        \input caterpillar2.tex
    \end{figure}
    
    The idea of this construction is making almost everything symmetric, but using the only spine vertex with rank two to "break the balance". We need to check that, under our $f$, each vertex has a distinct string.

    Firstly, two spine vertices have different strings. 
    \begin{itemize}
        \item If $p\neq q$ and $p\neq n+1-q$, then $s_p$ and $s_q$ have different strings. In the string of $s_x$, the first $\max\{x,\ n+1-x\}$ coordinates are nonzero, and the rest are zero. So if $p\neq q$ and $p\neq n+1-q$, then $s_p$ and $s_q$ have different numbers of nonzero coordinates in their strings.
        \item For any $1\le k\le \lfloor\frac{n}{2}\rfloor$, the string of $s_k$ and the string of $s_{n+1-k}$ are different. This is obvious because $s_k$ and $s_{n+1-k}$ have different distances from $s_1$, the only spine vertex with rank two. For example, in Figure \ref{caterpillar2}, the rank of $s_1$ is counted in the second coordinate in the string of $s_3$; but it is counted in the third coordinate in the string of $s_4$. Actually, the string of $s_3$ is $(5,\ 16,\ 14,\ 3,\ 0,\ 0,\ 0)$, and the string of $s_4$ is $(5,\ 15,\ 15,\ 3,\ 0,\ 0,\ 0)$.
    \end{itemize}

    Secondly, two attached leaves have different strings. For two leaves $u$ and $v$:
    \begin{itemize}
        \item If they are attached on the same spine vertex, then the second coordinates in their strings are different, because $f(u)\neq f(v)$, and $f(u)$ contributes to the second coordinate in the string of $v$, and vice versa.
        \item If $u$ is attached on $s_p$, $v$ is attached on $s_q$, where $p\neq q$ and $p\neq n+1-q$, then they have different numbers of nonzero coordinates in their strings, just like what we have for the spine vertices.
        \item If $u$ is attached on $s_k$, $v$ is attached on $s_{n+1-k}$, for some $1\le k\le \lfloor\frac{n}{2}\rfloor$, then we know their strings are different because $u$ and $v$ have different distances from $s_1$, the only spine vertex with rank two.
    \end{itemize}

    Also, a spine vertex and an attached leaf have different strings. 
    \begin{itemize}
        \item For a leaf $u$ attached on $s_j$ with $2\le j\le n$, the first coordinate in the string of $u$ is one, because $s_j$ is the only neighbor of $u$. For any spine vertex $s_i$, the first coordinate in the string of $s_i$ is at least two, because $s_i$ has at least two neighbors. So the string of $u$ must be different from the string of any $s_i$, as their first coordinates vary.
        \item If a leaf $v$ is attached on $s_1$, then we can see that every coordinate in the string of $v$ is nonzero. We can also see that the $(n+1)$-th coordinate in the string of $s_i$ is zero, for any $1\le i\le n$. So the string of $v$ must be different from the string of any $s_i$, as they have different numbers of nonzero coordinates.
    \end{itemize}

    \textbf{Case 2.} $L=1$.

    In this case, we have $IDI(G)\le 2$, because we can construct a rank assignment $f$ by letting $f(s_1)=2$ and $f(v)=1$ for $v\in V$ and $v\neq s_1$. It is easy to check that this assignment does the same trick as what we have in Case 1.

    Also, we have $IDI(G)\neq 1$, because if every vertex gets the same rank, then it is clear that the string of $s_1$ is the same as the string of $s_n$, as the caterpillar is symmetric. Thus, we have $IDI(G)=2$ if $L=1$.
\end{proof}

Determining the ID-indices of non-symmetric caterpillars might be taken as a future goal.

\begin{problem}
    Assuming $G$ is a non-symmetric caterpillar with $n$ spine vertices, which means we have $L_j\neq L_{n+1-j}$ for some $1\le j\le \lfloor\frac{n}{2}\rfloor$, find $IDI(G)$.
\end{problem}

\end{document}